\documentclass[12pt, a4paper,pdftex]{amsart} 

\title{Matroidal representations of groups} 

\author{Noah Giansiracusa}
\author{Jacob Manaker} 
\date{\today}

\usepackage{mathptmx}
\DeclareSymbolFont{cmlargesymbols}{OMX}{cmex}{m}{n}
\DeclareMathSymbol{\mycoprod}{\mathop}{cmlargesymbols}{"60}

\usepackage[protrusion=true,expansion=true]{microtype}

\usepackage{hyperref}
\usepackage{amssymb}
\usepackage{mathrsfs}  
\usepackage{latexsym}
\usepackage{amsmath}
\usepackage{color}
\usepackage{enumitem}
\usepackage{graphicx}
\usepackage{verbatim}

\usepackage[mathcal]{eucal}
\DeclareFontFamily{OT1}{pzc}{}
\DeclareFontShape{OT1}{pzc}{m}{it}{<-> s * [1.10] pzcmi7t}{}
\DeclareMathAlphabet{\mathpzc}{OT1}{pzc}{m}{it}

\usepackage{bbm} 

\usepackage[cmtip,arrow]{xy}
\usepackage{pb-diagram,pb-xy}
\usepackage{tikz}
\tikzset{bdot/.style={circle, thick, draw=black, top color=black!20, bottom color=black}, 
              wdot/.style={circle, thick, draw=black!20, top color=white, bottom color=white}}

\usepackage[vmargin=3.0cm, hmargin=3.3cm]{geometry}
\numberwithin{equation}{subsection}

\newtheorem{theorem}{Theorem}[subsection]  

\newtheorem{lemma}[theorem]{Lemma} 
\newtheorem{proposition}[theorem]{Proposition}
\newtheorem{corollary}[theorem]{Corollary}
\newtheorem{conjecture}[theorem]{Conjecture}

\newtheorem{thmx}{Theorem}

\theoremstyle{remark} 
\newtheorem{remark}[theorem]{Remark}
\newtheorem{definition}[theorem]{Definition}
\newtheorem{example}[theorem]{Example}

\newcommand{\Hom}{\mathrm{Hom}}
\newcommand{\Aut}{\mathrm{Aut}}
\newcommand{\GL}{\mathrm{GL}}

\newcommand{\B}{\mathbb{B}}

\newcommand{\C}{\mathbb{C}}
\newcommand{\R}{\mathbb{R}}
\newcommand{\N}{\mathbb{N}}
\newcommand{\Z}{\mathbb{Z}}

\newcommand{\PP}{\mathbb{P}}
\newcommand{\T}{\mathbb{T}}

\newcommand{\trop}{\mathpzc{trop}}

\newcommand{\ord}{\operatorname{ord}}

\newcommand{\ext}{{\bigwedge}}
\newcommand*{\transpose}{\,\raisebox{4pt}{\ensuremath{\intercal}}}
\newcommand{\field}{\mathbbm{k}}
\newcommand{\semifield}{\mathbb{S}}

\makeatletter
\newcommand*\incircbin{\mathpalette\@incircbin}
\newcommand*\@incircbin[2]{\mathbin{\ooalign{\hidewidth$#1#2$\hidewidth\crcr$#1\ocircle$}}}
\newcount\dindex
\newcommand*\dsh[1][1]{\dindex=1%
\loop\ensuremath{\dabar@}%
\ifnum\dindex<#1%
\advance\dindex by 1%
\repeat%
}
\newcommand*\vblock{\rotatebox[origin=c]{90}{\ensuremath{\dabar@\dabar@\dabar@}}}
\makeatother

\newcommand{\ALERT}[1]{{\color{red}[#1]}}

\begin{document}

\maketitle

\begin{abstract}
We develop the rudiments of a finite-dimensional representation theory of groups over idempotent semifields by considering linear actions on tropical linear spaces.  This can be considered a tropical representation theory, a characteristic one modular representation theory, or a matroidal representation theory---and we draw from all three perspectives.  After some general properties and constructions, including a weak tropical analogue of Maschke's theorem, we turn to a study of the regular representation of a finite group and its tropicalization.  For abelian groups we find an interesting interplay between elementary number theory and matroid theory---even cyclic groups are surprisingly rich---and we conclude with some possible first steps toward a tropical character theory.
\end{abstract}

\section{Introduction}

The algebraic approach to tropical geometry and matroid theory that has been emerging in recent years \cite{Frenk,GG1,GG3,Maclagan-Rincon-ideals,Lorscheid-tropical,Jaiung-hyperrings,Crowley-Giansiracusa-Mundinger,Baker-Bowler} provides a framework in which to develop the rudiments of a ``characteristic one'' modular representation theory, meaning representation theory over an idempotent (i.e., $1+1=1$) semifield $\semifield$.  The goal of this paper is to embark on such a journey.  

The idea is to consider linear group actions on tropical linear spaces (equivalently, $\semifield$-valued matroids), not just free modules.  This allows for the incorporation of constructions from tropical geometry---such as the stable intersection and stable sum, as well as the tropicalization of certain classical representations over a field---and it provides a novel setting in which group theory and matroid are intriguingly interwoven.  Our main results and examples concern the regular representation for finite groups $G$ and show that both the matroidal structure of classical representations and the intrinsic $G$-invariant tropical linear geometry reflect significant group-theoretic structure of $G$; even the case of cyclic groups appears surprisingly rich.  We believe we are only scratching the surface here and that there is much more to be discovered in this direction.

\subsection{Summary of results}

Every linear automorphism of a free module over a semifield $\semifield$ is given by a \emph{monomial} matrix (sometimes called a \emph{generalized permutation} matrix), so a group homomorphism $G \rightarrow \GL(\semifield^n)$ (which we will call a \emph{linear representation} of $G$ over $\semifield$) is a mild generalization of a permutation representation in which we are allowed to rescale the basis vectors.  Given such a homomorphism, we define a \emph{tropical representation} to be a $G$-invariant tropical linear space in $\semifield^n$.  The following summarizes our basic results on this setup; no individual item on this list is particularly difficult or deep, but in tandem this provides a healthy beginning to the subject.

\newpage

\begin{thmx}
Fix a group $G$, idempotent semifield $\semifield$, and linear representation $G \rightarrow \GL(\semifield^n)$.
\begin{enumerate}
\item For any $1 \le d \le n$, the induced linear representation on $\ext^d\semifield^n$ restricts to a linear action on the Dressian $Dr(d,n) \subseteq \PP(\ext^d\semifield^n)$, and $d$-dimensional tropical representations in $\semifield^n$ are equivalent to $G$-fixed points in $Dr(d,n)$.
\item If $\semifield = \{0,1\}$, a tropical representation $T\subseteq \semifield^n$ is equivalent to a homomorphism from $G$ to the automorphism group of the matroid corresponding to $T$.
\item The stable sum and stable intersection of tropical representations are tropical representations, and if $\semifield$ is totally ordered then the perpendicular (i.e., valuated matroid dual) of a tropical representation is a tropical representation. 
\item Given an $n$-dimensional monomial representation $V$ of $G$ over a valued field $\nu : k \rightarrow \semifield$, matrix entry-wise valuation yields a linear representation over $\semifield$ and any subrepresentation of $V$ tropicalizes to a tropical representation in $\semifield^n$; this tropicalization operation is $G$-equivariant.
\end{enumerate}
\end{thmx}

With this setup in hand, we prove a weak tropical analogue of Maschke's Theorem when $G$ is a finite group: For any tropical representation $T\subseteq \semifield^n$, there is a tropical representation in $\semifield^n$ of the complementary dimension obtained by averaging the standard inner product on $\semifield^n$ over $G$ and applying the induced isomorphism $(\semifield^n)^\vee \cong \semifield^n$ to the perpendicular tropical linear space $T^\perp \subseteq (\semifield^n)^\vee$.

We next turn to a study of the regular representation $\C[G]$ of a finite group $G$ and its constant-coefficient tropicalization $\B[G]$, where $\B = \{0,1\}$ is the idempotent Boolean semifield alluded to above in item (2).  In this context we call a tropical representation in $\B[G]$ \emph{realizable} if it is the tropicalization of a subrepresentation of $\C[G]$.  The following summarizes our results in this direction; some of these require significantly more sophisticated methods than the results listed above.  In what follows we blur the distinction between a tropical linear space and the corresponding matroid.
\begin{thmx}
Fix a finite group $G$.  All tropical representations considered below are in the Boolean regular representation $\B[G]$.
\begin{enumerate}
\item For any $1\le d \le |G|$, the uniform matroid $U_{d,G}$ of rank $d$ on ground set $G$ is a tropical representation.
\item The only dimension one and codimension one tropical representations are uniform matroids; both of these are realizable.
\item If $G$ is abelian, then $G$ is cyclic if and only if the tropical representation $U_{2,G}$ is realizable.
\item If $G\cong\Z_p$ is cyclic of prime order, then for each $1 \le d \le p$ the uniform matroid $U_{d,G}$ is realizable; moreover, these uniform matroids are the only realizable tropical representations.
\item Let $p$ be an odd prime satisfying the following condition: either 2 is a primitive root of $p$, or $p \equiv 7~(\mathrm{mod}~8)$ and 2 has order $\frac{p-1}{2}$ in $\Z_p^*$.  Then the only two-dimensional representation of $\Z_p$ is $U_{2,\Z_p}$.
\end{enumerate}
\end{thmx}

Items (1) and (2) are straightforward to prove.  Item (3) relies on classical character theory.  Item (4) follows almost immediately from Chebotarev's theorem (originally Ostrowski's conjecture) on roots of unity.  We also adapt a modern proof of Chebotarev's theorem \cite{Vandermonde} to produce a sufficient condition for a given basis to exist in the matroid associated to a given subrepresentation of $\B[\Z_{p^r}]$.  The hypothesis for item (5) is an artifact of our proof technique: $p=17$ is the first prime not satisfying the stated condition and we checked using SAGE that in dimension two $\B[\Z_{17}]$ only has the uniform matroid.  On the other hand, $\Z_4$ and $\Z_6$ have non-uniform representations in dimension two.  Based on these observations, we conjecture that a cyclic group $\Z_n$ has prime order if and only if $U_{2,\Z_n}$ is the only two-dimensional tropical representation in $\B[\Z_n]$.

We do not know if a reasonable notion of irreducibility exists for tropical representations.  However, in the context of a fixed linear representation of $G$, such as the Boolean regular representation $\B[G]$, we can call a tropical representation \emph{indecomposable} if it is not the stable sum of two representations of smaller dimension.  Since the stable sum corresponds to the matroid union, and the matroid union of uniform matroids is uniform, it follows from item (2) above that every non-uniform tropical representation in dimension two is indecomposable.  Indecomposability in higher dimensions is more complicated.

Using the preceding results and direct ad hoc calculations, we study the Boolean regular representation for groups of order at most six, finding in particular the following. 
\begin{enumerate}
\item[$\Z_4:$] There is one non-uniform tropical representation in dimension two; it is realized by the direct sum of the trivial and sign representations.
\item[$\Z_2\times \Z_2$:] Since this is not cyclic, $U_{2,\Z_2\times\Z_2}$ is not realizable.  There are three non-uniform tropical representations in dimension two; all three of these are realizable.
\item[$\Z_5:$] The only tropical representation in each dimension is the uniform matroid.
\item[$\Z_6:$] There are two non-uniform tropical representations in dimension two; both of these are realizable.  There are two non-uniform representations in dimension three, one of which is indecomposable.
\item[$\Sigma_3:$] There are four non-uniform tropical representations in dimension two, all of which are realizable.  One of these is realized by both the reducible representation $V_{\text{triv}}\oplus V_{\text{sign}}$ and the irreducible representation $V_{std}$.  The uniform matroid $U_{2,\Sigma_3}$ is also realizable, showing that the abelian hypothesis in item (3) above is essential.  
\end{enumerate}

We conclude the paper by offering a definition of a tropical character and computing a proposed character table for $\Z_4$, $\Z_2\times \Z_2$, and $\Z_5$.  The challenge in finding a tropical analogue of character theory is that tropical linear spaces are not free modules so it's not clear what trace should mean.  However, tropical linear spaces do have a canonical generating set, the valuated cocircuit vectors, and in the Boolean case we show that for a tropical representation there is an induced permutation representation on the set of cocircuits so one can take the trace of the matrices coming from this latter representation.  We don't have a notion of irreducibility in the tropical setting, but we do know that classically every irreducible representation appears in the regular representation, so we suggest that the tropical character table of a finite group $G$ should be the table of character values for all tropical representations in $\B[G]$, and this is what we compute for the aforementioned three small groups.

\vspace{0.2in}
\emph{Acknowledgements}.  N.G. is supported by NSF DMS-1802263 and enthusiastically thanks his students Mads Shoraka and John Fan for their help exploring Lemma \ref{lem:MO}.

\section{Idempotent linear algebra}

Let $\semifield$ denote an idempotent semifield---that is, $s+s=s$ for all $s\in\semifield$, and all the axioms of a field are satisfied except for the existence of additive inverses; as usual, we denote the additive and multiplicative neutral elements by $0$ and $1$, respectively.  Let \[\B := \{0,1\} \subseteq \semifield\] denote the 2-element Boolean subsemifield; this is the initial object in the category of idempotent semirings.

There is a natural partial order on any $\semifield$-module (and hence on $\semifield$ itself), namely \[a \le b \Leftrightarrow a+b=b.\]  Many notions in tropical geometry readily extend from the max-plus algebra \[\T := \R\cup\{-\infty\}\] to an arbitrary \emph{totally} ordered idempotent semifield; however, one must be especially careful when passing to an arbitrary partially ordered idempotent semifield.  A helpful reference here is Frenk's thesis \cite{Frenk}.  For instance, he shows \cite[Proposition 3.1.1 and Corollary 3.1.3]{Frenk} that $\semifield^\times = \semifield\setminus\{0\}$ is a lattice-ordered abelian group and every lattice-ordered abelian group is of this form.

\subsection{Tropical linear spaces}

Here we briefly recall the algebraic (idempotent module-theoretic) perspective of tropical linear spaces.  Helpful background references include \cite{Speyer-tropical-linear,Speyer-Sturmfels-tropical-grassmannian,Fink-Rincon,Frenk}; in this paper we primarily use the notation and results in \cite{GG3}, which works over an arbitrary idempotent semifield $\semifield$ and introduces the exterior algebra formalism, and \cite{Crowley-Giansiracusa-Mundinger}, which focuses on the Booleans $\B$ and elaborates the tight interplay with matroid theory that arises in that case.

First, recall that quotients in the idempotent setting are given by congruences (see, e.g., \cite[\S2.4]{GG1}) rather than submodules or ideals, since an identification $f\sim g$ cannot usually be achieved by one of the form $f-g\sim 0$ due to the lack of additive inverses.  Given a free module $V\cong \semifield^n$ with basis $e_1,\ldots,e_n$, the tropical exterior algebra $\ext V = \bigoplus \ext^d V$ is the quotient of the symmetric algebra by the congruence generated by $e_i^2 \sim 0$ for all $1 \le i \le n$.  It is \emph{not} true here, as it is classically over a field, that the wedge square of any element in an exterior power is zero.  We write $e_I := e_{i_1}\wedge \cdots \wedge e_{i_d}$ for $I=\{i_1,\ldots,i_d\}$; then the $e_I$ for $I\in\binom{[n]}{d}$ form a basis for the free module $\ext^d V$.  We denote the dual module by \[V^\vee := \Hom(V,\semifield)\] and the dual basis to the $e_i$ by $x_i$; then $\ext^d (V^\vee) \cong (\ext^d V)^\vee$ and the dual basis to the $e_I$ is given by $x_I := x_{i_1}\wedge \cdots \wedge x_{i_d}$.

A rank $d$ \emph{tropical Pl\"ucker vector} is a nonzero multivector $v = \sum v_Ie_I \in \ext^d V$ satisfying the \emph{tropical Pl\"ucker relations}:
\[\sum_{i\in A\setminus B}v_{A\setminus i} v_{B\cup i} = \sum_{i\in (A\setminus B)\setminus j}v_{A\setminus i} v_{B\cup i}\] for all $A\in \binom{[n]}{d+1}$, $B\in \binom{[n]}{d-1}$, and $j\in A\setminus B$. (This is the ``bend relations'' formalism introduced in \cite{GG1} that extends the usual  ``max attained at least twice'' condition beyond the totally ordered case, cf. \cite[\S4.1]{GG3}.)  

To each tropical Pl\"ucker vector $v$ there is an associated submodule $T_v \subseteq V$, called a \emph{tropical linear space}, which is the tropical kernel (see \cite[\S2.4]{GG3} for the definition of tropical kernel, we shall not need it in this paper) of the wedge multiplication map $-\wedge v : V \rightarrow \ext^{d+1} V$; concretely, $T_v$ is the intersection of the tropical hyperplanes defined by the \emph{valuated circuits} \[\sum_{i\in A} v_{A\setminus i}x_i \in V^\vee\] for all $A\in\binom{[n]}{d+1}$, and when $\semifield$ is totally ordered this can be equivalently described as the submodule spanned by the \emph{valuated cocircuits} \[\sum_{i\in [n]\setminus B}v_{B\cup i}e_i \in V\] for all $B\in\binom{[n]}{d-1}$.  See \cite[\S4.4]{GG3} for more details.  Tropical linear spaces only depend on the projective equivalence class of tropical Pl\"ucker vectors, so we tend to view the latter in the projective space $\PP(\ext^d V)$.  The tropical Pl\"ucker vector $v\in \PP(\ext^d V)$ can be recovered from the tropical linear space $T_v \subseteq V$ by taking an appropriate top wedge power \cite[Proposition 6.2.1]{GG3}.

When $\semifield=\B$, a tropical Pl\"ucker vector is precisely the basis-indicator vector of a matroid (the tropical Pl\"ucker relations reformulate the strong basis exchange axiom) and the associated tropical linear space is the module spanned by the indicator vectors of the cocircuits.

 \subsection{Inner products}
 
When later discussing a tropical analogue of Maschke's theorem we will need to average an inner product over a finite group, in very close analogy with the classical story.  While doing so we will rely on the following general result.  First, some terminology: for $V\cong \semifield^n$, we call a symmetric bilinear form \[\langle -,-\rangle : V\times V \rightarrow \semifield\] \emph{non-degenerate} if the induced linear map $V \rightarrow V^\vee$, $v \mapsto \langle v, -\rangle$, is an isomorphism.
 
\begin{proposition}\label{prop:nondegen}
If $\langle -,-\rangle$ is any symmetric bilinear form on  $V$ such that the induced map $V \rightarrow V^\vee$ is injective, then this form is non-degenerate.
\end{proposition}

\begin{proof}
Let us denote the induced map $V \rightarrow V^\vee$ by $\varphi$.  Then we are assuming $\varphi$ is injective and we must show that it is surjective; so let $f\in V^\vee$.  By \cite[Lemma 3.2.1(ii) and subsequent Remark]{Frenk}, the dual map \[V \cong V^{\vee\vee} \xrightarrow{\varphi^\vee} V^\vee\] is surjective.  This means there exists $v\in V$ such that $\varphi^\vee(ev_v)  = f$, where $ev_v : V^\vee \rightarrow \semifield$ is the linear form given by evaluating linear forms at $v$.   Unpacking this, we have that $f = ev_v\circ \varphi$ is the linear form $w \mapsto \langle w, v \rangle$.  By symmetry this equals the linear form $w \mapsto \langle v, w\rangle$, namely $\varphi(v)$, so $f = \varphi(v)$.
\end{proof}

\section{Tropical representations}

Throughout this section we let $G$ denote an arbitrary group and we denote by $V \cong \semifield^n$ a free module of rank $n \in \N$ over the idempotent semifield $\semifield$. By \cite[Proposition 2.2.2]{GG3}, a basis for $V$ is unique up to permutation and rescaling; put another way, the group of linear automorphisms is \[\GL(V) \cong \Sigma_n \ltimes (\semifield^\times)^n\] and relative to a fixed basis it consists of \emph{monomial matrices} (i.e., permutation matrices where the nonzero entries are allowed to be any nonzero value, not just 1).

\begin{definition}
A \emph{linear representation} of $G$ over $\semifield$ is a group homomorphism $G \rightarrow \GL(V)$.  A \emph{tropical subrepresentation} is a $G$-invariant tropical linear space $T \subseteq V$.
\end{definition}

We shall use the term \emph{tropical $G$-representation}, or simply \emph{representation} if the context is clear, to refer to either a linear representation of $G$ over $\semifield$ or a tropical subrepresentation (the former is a special case of the latter since $V$ is a tropical linear subspace of itself); in either case, by the \emph{dimension} of the representation we shall mean the rank of the tropical linear space, which is also the rank of the corresponding matroid.

Given a linear representation $G \rightarrow \GL(V)$, there is an induced linear representation on any tensor, symmetric, or exterior power of $V$.  (The tensor and symmetric powers of a module over an idempotent semifield are defined as usual, whereas the exterior algebra is more subtle and is studied in \cite{GG3}.)  The induced representation $G \rightarrow \GL(\ext^d V)$, for $1 < d < n$, is particularly important for us:

\begin{proposition}\label{prop:DressAct}
For a linear representation $V$ of $G$, if $v\in \ext^d V$ satisfies the tropical Pl\"ucker relations then so does $g\cdot v$ for each $g\in G$.  Thus the linear $G$-action on $\ext^d V$ restricts to an action on the Dressian $Dr(d,n) \subseteq \PP(\ext^d V)$.
\end{proposition}

\begin{proof}
Suppose $v = \sum v_Ie_I\in\ext^d V$ satisfies the tropical Pl\"ucker relations, so that \[\sum_{i\in A\setminus B}v_{A\setminus i} v_{B\cup i} = \sum_{i\in (A\setminus B)\setminus j}v_{A\setminus i} v_{B\cup i}\] for all $A\in \binom{[n]}{d+1}$, $B\in \binom{[n]}{d-1}$, and $j\in A\setminus B$, and that $g \in G$ is represented by a matrix in $\GL(V) \cong \Sigma_n\ltimes (\semifield^\times)^n$ that is the product of a permutation $\sigma$ and a diagonal matrix with diagonal entries $\lambda_1,\ldots,\lambda_n$,  Then \[g\cdot v = \sum \lambda_I v_Ie_{\sigma(I)} = \sum \lambda_{\sigma^{-1}(I)}v_{\sigma^{-1}(I)}e_I,\] where $\lambda_J := \prod_{i\in J}\lambda_i$.  Thus the tropical Pl\"ucker relations hold for $g\cdot v$ if for any $C\in\binom{[n]}{d+1}$ and $D\in\binom{[n]}{d-1}$, any term in the following sum can be omitted without changing the value of the sum: \[\sum_{i\in C\setminus D}\lambda_{\sigma^{-1}(C\setminus i)}\lambda_{\sigma^{-1}(D\cup i)}v_{\sigma^{-1}(C\setminus i)} v_{\sigma^{-1}(D\cup i)}.\]  Let $A := \sigma^{-1}(C)$ and $B := \sigma^{-1}(D)$.  Since summing over \[i\in C\setminus D = \sigma(A)\setminus\sigma(B)\] is equivalent to summing over \[j=\sigma^{-1}(i) \in A\setminus B,\] the preceding sum is equal to
\[\sum_{j\in A\setminus B}\lambda_{A\setminus j}\lambda_{B\cup j}v_{A\setminus j}v_{B\cup j} = \lambda_A\lambda_B \sum_{j\in A\setminus B}v_{A\setminus j}v_{B\cup j},\] so we can indeed omit any individual term by the hypothesis that $v$ satisfies the tropical Pl\"ucker relations.
\end{proof}

The $G$-action on $V$ and its tropical linear submodules is compatible with the $G$-action on tropical Pl\"ucker vectors:

\begin{proposition}
Fix a linear representation $V$ of $G$.  If $T_v \subseteq V$ is a tropical linear space with tropical Pl\"ucker vector $v \in \ext^d V$, then the tropical linear space associated to the tropical Pl\"ucker vector $g\cdot v$ is $g\cdot T_v \subseteq V$:
\[T_{g\cdot v} = g\cdot T_v.\]
\end{proposition}

\begin{proof}
The submodule $g\cdot T_v \subset V$ is a tropical linear space of rank $d$.  If we show that $g\cdot T_v$ is contained in the tropical hyperplane defined by each valuated circuit vector of the rank $d$ tropical Pl\"ucker vector $g\cdot v$, then we will have that $g\cdot T_v \subseteq  T_{g\cdot v}$; we would then be done since a containment of tropical linear spaces of the same rank must be an equality (by, e.g., \cite[Corollary 5.3]{Crowley-Giansiracusa-Mundinger} together with \cite[Corollary 7.3.4]{Oxley}).  So we will establish this assertion about valuated circuit vectors.

Let $p = \sum p_i e_i\in T_v$ and, as in the proof of Proposition \ref{prop:DressAct}, let $g\in G$ be represented by the product of a permutation $\sigma$ and a diagonal matrix with entries $\lambda_1,\ldots,\lambda_n$---and for any $J \subseteq [n]$ write $\lambda_J := \prod_{i\in J}\lambda_i$.  Then \[g\cdot p = \sum \lambda_i p_i e_{\sigma(i)} = \sum \lambda_{\sigma^{-1}(i)}p_{\sigma^{-1}(i)}e_i.\]  On the other hand, as we saw in the proof of Proposition \ref{prop:DressAct}, for $v=\sum v_Ie_I$ we have \[g\cdot v = \sum \lambda_{\sigma^{-1}(I)}v_{\sigma^{-1}(I)}e_I.\]  Thus for each $A\in\binom{[n]}{d+1}$ we have the following valuated circuit vector of $g\cdot v$:
\[\sum_{i\in A}\lambda_{\sigma^{-1}(A\setminus i)}v_{\sigma^{-1}(A\setminus i)}x_i.\]
Plugging $g\cdot p$ into this valuated circuit yields
\[\lambda_{\sigma^{-1}(A)}\sum_{i\in A} v_{\sigma^{-1}(A\setminus i)}p_{\sigma^{-1}(i)}.\]
Setting $C := \sigma^{-1}(A)$ and summing over $j=\sigma^{-1}(i) \in \sigma^{-1}(A) = C$ shows that this equals $\lambda_C$ times the valuated circuit vector of $v$ associated to the set $C$ evaluated at $p$, which means, since $p\in T_v$, we can delete any individual term without changing the value of the sum, which is what we needed to show.
\end{proof}

\begin{corollary}\label{cor:DressAct}
Dimension $d$ tropical subrepresentations of a linear representation are the same as fixed points of the $G$-action on the Dressian $Dr(d,n)$.
\end{corollary}

\subsection{Matroidal interpretation}\label{sec:matroidinter}

When $\semifield = \B$, a tropical linear space $T\subseteq V$ is equivalent to a matroid $M(T)$ on the ground set $[n] := \{1,\ldots,n\}$ and we have:

\begin{proposition}\label{prop:matroid}
A tropical $G$-representation $T\subseteq V$ is equivalent to a group homomorphism $G \rightarrow \Aut(M(T))$ to the matroid automorphism group.
\end{proposition}

\begin{proof}
Since $\B$ only has two elements we have $\B^\times = \{1\}$, so a linear representation $G \rightarrow \GL(V)$ is equivalent to a $G$-action on $[n]$, namely a homomorphism $G \rightarrow \Sigma_n$.  By definition, an automorphism of a matroid on $[n]$ is a permutation of $[n]$ that sends independent sets to independent sets.  The latter is equivalent to the permutation sending cocircuits to cocircuits, by Lemma \ref{lem:MatAut} below.  Thus $G \rightarrow \GL(V) \cong \Sigma_n$ factors through $\Aut(M(T))$ if and only if $T$ is $G$-invariant.
\end{proof}

\begin{remark}
Thus, over $\B$, we may think of a tropical $G$-representation as a linear $G$-action on a tropical linear space or equivalently as a $G$-action on a matroid.
\end{remark}

\begin{lemma}\label{lem:MatAut}
A permutation is an automorphism of a matroid if and only if it sends cocircuits to cocircuits.
\end{lemma}

\begin{proof}
This is elementary, but we include a proof for the convenience of those unfamiliar with matroids.  Suppose $\sigma\in\Sigma_n$ is an automorphism of a matroid $M$, i.e., it sends independent sets to independent sets.  Then it must send dependent sets (that is, sets which are not independent) to dependent sets, otherwise $\sigma^{-1}$, which is also an automorphism, would send an independent set to a dependent set.  Moreover, $\sigma$ must send circuits (minimal dependent sets) to circuits, for if $C\subseteq [n]$ is a circuit then we know $\sigma(C)$ is dependent and if it properly contained a dependent subset, say $D \subsetneq \sigma(C)$, then $\sigma^{-1}(D)$ would be a proper dependent subset of $C$, contradicting the fact that $C$ is a circuit.  Finally, the circuits of $M$ are precisely the cocircuits of the dual matroid $M^*$, and a permutation is an automorphism of a matroid if and only if it is an automorphism of the dual matroid, since at the level of bases (maximal independent sets) matroid duality is simply set-theoretic complement.  Thus an automorphism of $M$ is an automorphism of $M^*$ and hence sends circuits of $M^*$ to circuits of $M^*$, i.e., cocircuits of $M$ to cocircuits of $M$.  The converse is similar.
\end{proof}

\subsection{Tropicalizing classical representations}

Let $\field$ be a field equipped with a valuation $\nu : \field \rightarrow \semifield$ (see \cite[\S2.5]{GG1} for a discussion of non-archimedean valuations when the idempotent semifield is not necessarily totally ordered).  \emph{Throughout this section we fix a choice of basis for $W\cong \field^n$.}  Tropicalization, with respect to this fixed basis, sends linear subspaces $L\subseteq W$ to tropical linear subspaces $\trop(L)\subseteq \semifield^n \cong V$, where $\trop(L)$ is the $\semifield$-module generated by $\{\nu(w)~|~w\in L\}$.  The tropical Pl\"ucker of $\trop(L)$ is the coefficient-wise valuation of the Pl\"ucker vector of $L$ \cite{Speyer-tropical-linear}.

Let us call a $\field$-linear representation $G \rightarrow \GL(W)$ a \emph{monomial representation} if the matrix for each $g\in G$, with respect to the \emph{fixed} basis of $W$, is a monomial matrix.  

\begin{remark}
One might naturally wonder how restrictive it is to require a representation to be monomial for some choice of basis.  A finite group for which every irreducible representation is monomial in some basis is called an M-group.  Every M-group is solvable \cite[Corollary 2.3.4]{Bray}, and conversely every supersolvable group is an M-group \cite[Corollary 2.3.5]{Bray}.
\end{remark}

\begin{lemma}
Applying the valuation $\nu$ matrix entry-wise yields a homomorphism from the group $\Sigma_n\ltimes (\field^\times)^n \subseteq \GL(W)$ of $n\times n$ monomial matrices over $\field$ to $\GL(V)$.
\end{lemma} 

\begin{proof}
This follows immediately from the facts that (1) valuations are multiplicative, and (2) the polynomial entries of the matrix multiplication map \[\GL(W) \times \GL(W) \rightarrow \GL(W),\] when restricted to the loci of monomial matrices, are monomials. 
\end{proof}

This lemma allows us to tropicalize monomial representations:

\begin{definition}
The \emph{tropicalization} of a monomial representation $\rho$ of a group $G$ over a field $\field$ is the linear representation over $\semifield$ defined by the composition \[G \xrightarrow{\rho} \Sigma_n\ltimes (\field^\times)^n \xrightarrow{\nu} \GL(V).\]
\end{definition}

\begin{proposition}
Fix a monomial representation $G \rightarrow \Sigma_n\ltimes (\field^\times)^n \subseteq \GL(W)$ and the corresponding tropicalized representation $G \rightarrow \GL(V)$ over $\semifield$.  If $L\subseteq W$ is any subspace, then for every $g\in G$ we have \[\trop(g\cdot L) = g \cdot \trop(L).\]
\end{proposition}

\begin{proof}
This also follows from the multiplicativity of valuations and the fact that we are only considering monomial representations.
\end{proof}

The preceding proposition shows that we can tropicalize subrepresentations:

\begin{corollary}
If $L\subseteq W$ is $G$-invariant for a monomial representation of $G$, then $\trop(L) \subseteq V$ is $G$-invariant for the tropicalization of this representation.
\end{corollary}

\begin{remark}
Using the translation from tropical to matroidal language discussed in \S\ref{sec:matroidinter}, the preceding corollary says that if $G$ acts monomially on a vector space $W$, then there is an induced action on the matroid associated to any subrepresentation $L\subseteq W$.
\end{remark}

We shall come back to the following notion in \S\ref{sec:regrep} when we study the regular representation of a finite group and its tropicalization:

\begin{definition}\label{def:realized}
For a monomial representation $W$ of $G$ and the corresponding tropicalized representation $V$, a tropical subrepresentation $T \subseteq V$ is \emph{realized} by a subrepresentation of $W$ if $T=\trop(L)$ for a subrepresentation $L \subseteq W$.
\end{definition}

\subsection{New representations from old ones}

While the tensor product of tropical linear spaces is not in general a tropical linear space (cf., \cite{LasVergnas}), and similarly for the exterior product (cf., \cite[Example 3.8]{Crowley-Giansiracusa-Mundinger}), tropical geometry has an interesting pair of constructions, the stable sum and the stable intersection, which do produce new tropical linear spaces from old ones.  We shall show here that stable operations preserve $G$-invariance of tropical linear spaces, and we construct an orthogonal complement operation for tropical linear spaces that also preserves $G$-invariance (and hence gives an analogue, though a fairly enigmatic one, of Maschke's theorem).

If $v\in\ext^d V$ and $v'\in\ext^{d'} V$ are tropical Pl\"ucker vectors such that $v\wedge v'\in \ext^{d+d'}V$ is nonzero, then the latter is a tropical Pl\"ucker vector \cite[Proposition 5.1.2]{GG3} and the corresponding tropical linear space $T_{v\wedge v'}$ is the \emph{stable sum} of $T_v$ and $T_{v'}$ \cite[Proposition 5.1.1]{GG3}.  

If we fix an identification $\ext^n V \cong \semifield$, say by choosing $e_{[n]}$ as the generator, then we have a tropical Hodge star \[\star : \ext^{n-d}V \cong \ext^d V^\vee\] which sends $e_I$ to $x_{[n]\setminus I}$.  If $\semifield$ is totally ordered then the Hodge star sends a valuated matroid to its dual and a tropical linear space to its orthogonal dual: 
\[T_{\star v} = T_v^\perp,\]
where by definition the orthogonal dual $M^\perp \subseteq V^\vee$ of a submodule $M\subseteq V$ is the intersection of all tropical hyperplanes in $M$ (where we view $V \cong V^{\vee \vee}$ as linear forms on $V^\vee$).  See \cite[\S4.4]{GG3}.  The \emph{stable intersection} of $T_v$ and $T_{v'}$ is $T_{\star(\star v \wedge \star v')}$.

\begin{lemma}\label{lem:star}
For a linear $G$-representation $V$ and any $v\in \ext^d V$ and $g\in G$, with regard to the induced linear actions on $\ext^d V$ and $\ext^{n-d} V^\vee$ we have: 
\[g \cdot \star v = \frac{1}{\det(g)} \star(g\cdot v).\]
\end{lemma}

\begin{proof}
By linearity, it suffices to check this on basis elements.  If we write the matrix in $\GL(V)$ representing $g\in G$ as a product of a permutation $\sigma$ and a diagonal matrix with diagonal entries $\lambda_1,\ldots,\lambda_n$, and for any $J\subseteq [n]$ we set $\lambda_J := \prod_{i\in J}\lambda_i$, then we have
\[g\cdot e_I = \lambda_I e_{\sigma(I)} \stackrel{\star}{\longmapsto} \lambda_I x_{[n]\setminus \sigma(I)},\]
whereas
\[\star e_I = x_{[n]\setminus I} \stackrel{g\cdot}{\longmapsto} \frac{1}{\lambda_{[n]\setminus I}}x_{\sigma([n]\setminus I)}.\]
The assertion then follows from the observations that (1) $\sigma([n]\setminus I) = [n]\setminus \sigma(I)$ and (2) $\det(g) = \lambda_{[n]} = \lambda_{I}\lambda_{[n]\setminus I}$.
\end{proof}

\begin{theorem}\label{thm:stable}
Fix a linear representation $G \rightarrow \GL(V)$.  The stable sum and stable intersection of tropical subrepresentations is a tropical subrepresentation, and if $\semifield$ is totally ordered then the orthogonal dual of a tropical subrepresentation is a tropical subrepresentation of the dual representation $G \rightarrow \GL(V^\vee)$.
\end{theorem}

\begin{proof}
The claim about the orthogonal dual follows immediately from the facts that (1) a tropical linear space is $G$-invariant if and only if its projectivized tropical Pl\"ucker vector is a fixed point for the induced action on the Dressian (Corollary \ref{cor:DressAct}), and (2) up to scalar, the Hodge star commutes with the $G$-action (Lemma \ref{lem:star}).

Now fix tropical Pl\"ucker vectors $v\in \ext^d V$ and $v'\in\ext^{d'} V$ with $v\wedge v' \ne 0$, and suppose $T_v$ and $T_{v'}$ are subrepresentations.  Also fix an element $g\in G$.  Using observation (1) above we can write $g\cdot v = \lambda v$ and $g\cdot v' = \lambda' v'$ for some nonzero scalars $\lambda,\lambda'\in\semifield^\times$.  We then have 
\[g\cdot (v\wedge v') = (g\cdot v)\wedge (g\cdot v') = (\lambda v)\wedge (\lambda' v') = \lambda\lambda'(v\wedge v'),\] so by applying observation (1) again we see that the stable sum is indeed $G$-invariant.

Finally, by using Lemma \ref{lem:star} twice more we have 
\begin{align*}
g\cdot \star(\star v\wedge \star v') &= \frac{1}{\det(g)} \star (g \cdot (\star v \wedge \star v'))\\
&= \frac{1}{\det(g)} \star ((g\cdot \star v)\wedge (g \cdot \star v'))\\
&= \frac{1}{\det(g)^3} \star (\star (g \cdot v)\wedge \star (g\cdot v'))\\
&= \frac{\lambda\lambda'}{\det(g)^3} \star (\star v\wedge \star v'),
\end{align*}
so once again applying observation (1) shows that the stable intersection is indeed $G$-invariant.
\end{proof}

We have seen that if $T \subseteq V$ is a tropical subrepresentation, then so is the orthogonal tropical linear space $T^\perp \subseteq V^\vee$; but the latter lives in the dual space and it would be useful to have some kind of complementary tropical linear space to $T$ that lives in $V$ yet is still $G$-invariant.  The way to accomplish this is with an isomorphism $V^\vee \cong V$, or equivalently a non-degenerate symmetric bilinear form, but we must be careful here to use a $G$-invariant form.  Fortunately, at least when $G$ is finite we can average the standard inner product over $G$-orbits quite analogously to the standard classical proof of Maschke's theorem that $G$-invariant subspaces admit $G$-invariant complements:

\begin{theorem}\label{thm:complement}
If $G$ is a finite group and $\langle -, - \rangle$ is the standard inner product on $V$, then \[\langle -, - \rangle_G := \sum_{g\in G} \langle g\cdot -, g\cdot -\rangle\] is a non-degenerate symmetric bilinear form.  If $\semifield$ is totally ordered and $T\subseteq V$ is a tropical subrepresentation, then the isomorphism $V^\vee \cong V$ induced by $\langle -, - \rangle_G$ sends $T^\perp$ to a tropical subrepresentation of $V$.
\end{theorem}

\begin{proof}
To show that $\langle -, - \rangle_G$ is non-degenerate, it suffices by Proposition \ref{prop:nondegen} to show that the induced map $V \rightarrow V^\vee$ sending $v = \sum v_i e_i$ to $\langle v, - \rangle_G$ is injective.  If we write each $g\in G$ as the product of a permutation $\sigma_g\in\Sigma_n$ and a diagonal matrix with diagonal entries $\lambda_{g,1},\ldots,\lambda_{g,n}$, then for each $1 \le j \le n$ we have 
\[\langle v, e_j \rangle_G = \sum_{g\in G} \left\langle \sum_i v_i\lambda_{g,i}e_{\sigma_g(i)}, \lambda_{g,j}e_{\sigma_g(j)} \right\rangle = \sum_{g\in G}v_j\lambda_{g,j}^2 = v_j\sum_{g\in G}\lambda_{g,j}^2.\]
In an idempotent semiring, if a sum is zero then all terms in the sum must be zero, so $\sum\lambda_{g,j}^2$ is a nonzero constant that depends only on the group $G$ and the representation, not on $v$.  Thus we can recover each coefficient $v_j$ from this pairing, which proves the desired injectivity.

It is clear from the definition that $\langle -, - \rangle_G$ is $G$-invariant.  It follows that the isomorphism $V \cong V^\vee$, $v\mapsto \langle v,-\rangle_G$, is $G$-equivariant.  Indeed, for $g\in G$ and $v\in V$ this $G$-invariance yields 
\[\langle g\cdot v, - \rangle_G = \langle v, g^{-1}\cdot -\rangle_G,\] and by the definition of the dual action, for any $\varphi\in V^\vee$ we have $g\cdot \varphi = \varphi \circ g^{-1}$, so \[g\cdot \langle v,-\rangle_G = \langle v, g^{-1}\cdot -\rangle_G.\]

For $T\subseteq V$ a $G$-invariant tropical linear space, we have shown in Theorem \ref{thm:stable} that the orthogonal tropical linear space $T^\perp \subseteq V^\vee$ is $G$-invariant with respect to the dual action.  Since the isomorphism $V^\vee \cong V$ is $G$-equivariant, the image of $T^\perp$ in $V$ is $G$-invariant; it is also a tropical linear space since any isomorphism between free $\semifield$-modules sends tropical linear spaces to tropical linear spaces.
\end{proof}

In some sense this theorem can be viewed as a tropical analogue of Maschke's theorem, since for any tropical subrepresentation it produces a subrepresentation of the complementary dimension.  However, such a pair of submodules very rarely gives a direct sum decomposition of $V$.  In fact, the following example shows that the complementary tropical linear space may be contained in the original tropical linear space.

\begin{example}
Consider the defining representation $W \cong \field^3$ of the symmetric group $G := \Sigma_3$ and its tropicalization $V \cong \semifield^3$.  The standard inner product on $V$ is $G$-invariant, by idempotency of addition, so the induced isomorphism $V^\vee \cong V$ sends $x_i$ to $e_i$.  On the classical side, the complement of the trivial representation is the regular representation, whereas on the tropical side the trivial representation is spanned by $e_1+e_2+e_3$ so its complement, the tropical linear space associated to the uniform matroid $U_{2,3}$, is spanned by $e_i+e_j$ for all $i\ne j$.  Therefore in this tropical setting the two-dimensional complement to the trivial representation contains the trivial representation.  
\end{example}

\begin{remark}
We do not know whether there is a useful notion of \emph{irreducibility} for tropical representations.  However, this may not be cause for concern (and indeed irreducibility is mysterious and largely absent from tropical geometry more generally).  To wit, we believe that tropical representation theory is really about the way invariant tropical linear spaces sit inside the ambient free module, rather than about the intrinsic module-theoretic structure of representations.  This perspective is explored in detail in the following section.
\end{remark}

\section{The regular representation}\label{sec:regrep}

Any permutation representation is a monomial representation and hence may be tropicalized.  A natural, and important, example of a permutation representation is the (classical) regular representation $\field[G]$ of a finite group $G$.  In this section we study its trivial-valuation tropicalization $\B[G]$---which could reasonably be termed the \emph{Boolean} or \emph{matroidal} regular representation.  

By Proposition \ref{prop:matroid}, the tropical subrepresentations of $\B[G]$ are precisely the matroids on ground set $G$ for which left-multiplication by each element of $G$ is a matroid automorphism.  When such a subrepresentation is realized by a subrepresentation of $\field[G]$, we shall simply say that the subrepresentation of $\B[G]$ is \emph{realizable}.  This means precisely that the matroid associated to the subrepresentation of $\B[G]$ is the matroid of the corresponding subrepresentation of $\field[G]$.  In other words, the realizable subrepresentations of the Boolean regular representation are the matroids of the subrepresentations of the classical regular representation.  In particular, since every irreducible $\field$-representation is naturally a subrepresentation of $\field[G]$, the matroid of every irreducible $\field$-representation (as embedded in the classical regular representation) appears in the Boolean regular representation $\B[G]$.

The Boolean regular representation is usually much larger than the locus of realizable subrepresentations.  For instance, even though every rank two matroid is representable (as the matroid of some linear subspace of a vector space), we will see below that realizability as a subrepresentation of the regular representation reflects some group-theoretic structure of $G$.

Throughout this section we fix a finite group $G$ and equip the field of complex numbers with the trivial valuation $\nu : \C \rightarrow \B$.  We begin by showing that nothing interesting happens in dimension and codimension one.

\begin{proposition}\label{prop:dimone}
The only dimension one subrepresentation of $\B[G]$ is the trivial representation, which corresponds to the uniform matroid $U_{1,|G|}$; this subrepresentation is always realizable.  There is only one subrepresentation of codimension one, and it also corresponds to the uniform matroid, namely $U_{|G|-1,|G|}$, and is always realizable.
\end{proposition}

\begin{proof}
Since a rank one tropical linear space is spanned by a single vector, and since $\B$ has only two elements, we see that a dimension one subrepresentation consists of the zero vector and a single nonzero vector $\sum v_g e_g \in \B[G]$ that is fixed by every $g\in G$.  Since each $v_g$ is either 0 or 1, the fact that at least one $v_g$ is 1 and that the left-multiplication action of $G$ on itself is transitive implies that every $v_g$ is 1.  This means we have the trivial representation $\langle \sum e_g \rangle$ which clearly corresponds to the uniform rank one matroid and is realized by the trivial representation in $\C[G]$.

This proves the assertions about dimension one, and the assertions about codimension one, except for realizability, then all follow immediately from Lemma \ref{lem:star} and the fact that the only nonzero element of $\B$ is equal to its own inverse.  We claim that $U_{|G|-1,|G|}$ is realized as a subrepresentation by the orthogonal complement to the trivial representation in $\C[G]$, namely $\{\sum w_g e_g \in W~|~ \sum w_g = 0\}$.  Indeed, a matrix defining this this latter subspace is
\[\left(\begin{array}{cccc}1 & 1 & \cdots & 1 \\-1 & 0 & \cdots & 0 \\0 & -1 & \ddots & \vdots \\ \vdots & 0 & \ddots & 0 \\0 & \vdots & \ddots & -1\end{array}\right)\]
which clearly has all $|G|$ maximal minors nonzero.
\end{proof}

Next, we have the following simple but important observation:
\begin{proposition}\label{prop:unif}
The tropical linear space corresponding to the uniform matroid $U_{d,|G|}$, for any rank $1 \le d \le |G|$, is a subrepresentation of $\B[G]$.
\end{proposition}

\begin{proof}
The tropical Pl\"ucker vector in question is 
\[\sum_{I\in \binom{G}{d}} e_I.\]  Since left-multiplication by $G$ permutes the elements of $G$, it also permutes the size $d$ subsets of $G$, and so this tropical Pl\"ucker vector is fixed by the $G$-action; the result then follows from Corollary \ref{cor:DressAct}.
\end{proof}

Realizability of these uniform subrepresentations outside the dimension one and codimension one cases, as we explore below, is a much subtler matter.

\subsection{Finite abelian groups}

In this subsection we let $G$ be a finite abelian group.  In this case all the irreducible representations of $G$ are one-dimensional.  Let us enumerate the group as $G = \{g_1,\ldots,g_s\}$ and denote the irreducible representations by $\rho_i : G \rightarrow \C^\times$ for $i=1,\ldots,s$.  Each such irreducible representation uniquely embeds as a subrepresentation of the regular representation $\C[G]$; concretely, $\rho_i$ is given by the line \[L_i := \langle\rho_i(g_1)e_{g_1^{-1}} + \cdots + \rho_i(g_s)e_{g_s^{-1}}\rangle.\]  A dimension $d$ subrepresentation of $\C[G]$ must then be the direct sum of $d \le s$ distinct $L_i$, say $L_{i_1},\ldots,L_{i_d}$, which means its Pl\"ucker vector is the vector of maximal minors of the matrix
\[\left(\begin{array}{ccc}\rho_{i_1}(g_1) & \cdots & \rho_{i_1}(g_s) \\ \vdots & \ddots & \vdots \\\rho_{i_d}(g_1) & \cdots & \rho_{i_d}(g_s)\end{array}\right).\]

\begin{theorem}\label{thm:cyclic}
A finite abelian group $G$ is cyclic if and only if the subrepresentation corresponding to the rank two uniform matroid $U_{2,|G|}$ is realizable.
\end{theorem}

\begin{remark}
The abelian hypothesis on $G$ cannot be omitted in this theorem.  For instance, we show in \S\ref{sec:ord6symmetric} that for the symmetric group on three letters the uniform two-dimensional subrepresentation is not realized by a direct sum of one-dimensional representations, but it \emph{is} realized by a sufficiently general copy of the (irreducible two-dimensional) standard representation.
\end{remark}

\begin{proof}
By the preceding discussion, the tropical linear space associated to $U_{2,|G|}$ is realized as a subrepresentation if and only if there exists a pair of distinct characters $\rho_a, \rho_b$ such that 
\[\det \left(\begin{array}{cc}\rho_a(g_i) & \rho_a(g_j) \\ \rho_b(g_i) & \rho_b(g_j)\end{array}\right) \ne 0\]
for all $i\ne j$.  Since 1-dimensional characters form a group, this determinant \[\rho_a(g_i)\rho_b(g_j) - \rho_a(g_j)\rho_b(g_i)\] is nonzero iff \[\rho_{ab^{-1}}(g_i) \ne \rho_{ab^{-1}}(g_j).\]  Therefore, realizability is equivalent to the existence of an injective irreducible character, i.e., a faithful 1-dimensional representation.  If $G$ is cyclic then sending a generator to a primitive $|G|$-th root of unity gives a faithful representation $G \hookrightarrow \C^\times$, and conversely any finite subgroup of $\C^\times$ is cyclic so if a faithful 1-dimensional representation exists then $G$ must be cyclic.
\end{proof}

\begin{theorem}\label{thm:prime2dim}
Let $p$ be an odd prime satisfying either of the following conditions: (i) 2 is a primitive root of $p$, or (ii) $p \equiv 7~(\mathrm{mod}~8)$ and 2 has order $\frac{p-1}{2}$ in $\Z_p^*$.  Then the only two-dimensional representation of the cyclic group $\Z_p$ is $U_{2,\Z_p}$.
\end{theorem}

\begin{proof}
We view $\Z_p$ as the set $\{0,\ldots,p-1\}$ and label the basis for $\B[\Z_p]$ accordingly.  Then the $\Z_p$-orbits in $\ext^2\B[\Z_p]$ are
\[f_i := \sum_{\substack{j < k \in \Z_p \\ k-j\in\{i,p-i\}}}e_{jk}\] for $i=1,\ldots,\frac{p-1}{2}$.  We need to show that the only nonzero $\B$-linear combination (i.e., sum) of these $f_i$ which satisfies the tropical Pl\"ucker relations (that is, yields the basis indicator vector for a matroid) is the sum of all the $f_i$.  We will do this by arguing that (1) if we include any particular $f_{a_1}$, then basis exchange requires us to include another, say $f_{a_2}$, and (2) the sequence $f_{a_1},f_{a_2},\ldots$ we obtain this way exhausts all the $f_i$.

Given an index $1 \le i \le \frac{p-1}{2}$, consider the basis $\{0,i\}$ provided by $f_i$.  Because all elements of $\Z_p$ appear in the remaining terms of $f_i$, basis exchange tells us that for any $m\in \Z_p\setminus\{0,i\}$, at least one of the following must also be a basis: $\{0,m\}$, $\{i,m\}$.  We need to consider two cases.  If $i$ is even, then set $m=\frac{i}{2}$.  Since $\{0,\frac{i}{2}\}$ and $\{\frac{i}{2},i\}$ both appear in $f_{\frac{i}{2}}$, basis exchange then requires us to include $f_{\frac{i}{2}}$.  On the other hand, if $i$ is odd then set $m = \frac{p+i}{2}$.  In this latter case we have $i < m < p$ and moreover \[m-i = p-m = \frac{p-1}{2}.\] This means both $\{0,\frac{p+i}{2}\}$ and $\{i,\frac{p+i}{2}\}$ appear in $f_{\frac{p-i}{2}}$, so basis exchange here requires us to include $f_{\frac{p-i}{2}}$.  

Therefore, if we start with any orbit $f_i$, applying this argument inductively gives us a sequence of orbits that basis exchange forces to be included as well.  By Lemma \ref{lem:MO} below, our assumption on $p$ guarantees that this sequence cycles through the entire list of orbits $f_1,\ldots, f_{\frac{p-1}{2}}$, and so the only linear combination of orbits that yields a matroid is the sum of all orbits, which yields the uniform matroid.
\end{proof}

Both the statement and proof of the following lemma are due to Gjergji Zaimi, who kindly answered our MathOverflow question \cite{MO-Collatz}.

\begin{lemma}\label{lem:MO}
Let $p$ be an odd prime and consider the function from $\{1,\ldots,\frac{p-1}{2}\}$ to itself defined by
\[ f(n) = 
  \begin{cases} 
      \frac{n}{2} & n\text{ even}\\
      \frac{p-n}{2} & n\text{ odd}.
   \end{cases}
\]
The orbit of any $n$ under iterations of this function is the entire set $\{1,\ldots,\frac{p-1}{2}\}$ if and only if the following holds: either 2 is a primitive root of $p$, or $p \equiv 7~(\mathrm{mod}~8)$ and 2 has order $\frac{p-1}{2}$ in $\Z_p^*$.
\end{lemma}

\begin{proof}
This function is a permutation and the question is whether it acts transitively, i.e., whether the orbit of 1 has size $\frac{p-1}{2}$.  It is equivalent to ask this question of the inverse permutation, which takes the form 
\[1 = 2^0 \mapsto \pm 2^1 \mapsto \pm 2^2 \mapsto \cdots \mapsto \pm 2^{\frac{p-1}{2}} = 1\text{ in }\Z_p^*\]
where the signs are uniquely determined by the condition that $\pm 2^r \in \{1,\ldots,\frac{p-1}{2}\}$ and the equality $2^{\frac{p-1}{2}} = 1$ comes from the fact that $\Z_p^*$ has order $p-1$ so $2^{\frac{p-1}{2}} = \pm 1$ in this group.  Thus our question is whether $2^r \not\equiv \pm 1 \text{ mod } p$ for all $0 < r < \frac{p-1}{2}$.

If 2 is a primitive root of $p$ then the above condition clearly holds: since 2 has order $p-1$ we know $2^{\frac{p-1}{2}} = -1$ and so no smaller power of 2 is $\pm 1$.  So assume $p \equiv 7~(\text{ mod }8)$ and $\ord(2) = \frac{p-1}{2}$.  The first condition says, by the two supplements to quadratic reciprocity, that 2 is a quadratic residue and $-1$ is not a quadratic residue.  Since 2 is a quadratic residue, all powers of 2 are quadratic residues, so -1 is not a power of two.  Since 2 has order $\frac{p-1}{2}$, we thus have $2^r \ne \pm 1$ for all $0 < r < \frac{p-1}{2}$, as desired.

Conversely, suppose $2^r \neq \pm 1$ in $\Z_p^*$ for all $0 < r < \frac{p-1}{2}$.  If $2^{\frac{p-1}{2}} = -1$, then $\ord(2) > \frac{p-1}{2}$, but $\ord(2)$ divides $|\Z_p^*|=p-1$ so it must equal $p-1$, i.e., 2 is a primitive root of $p$.  So let us assume $2^{\frac{p-1}{2}} = 1$.  Together with our hypothesis this implies that $-1$ is not a power of 2 and that $\ord(2) = \frac{p-1}{2}$.  The latter implies that 2 is the square of a primitive root, so then by Euler's criterion the powers of 2 cover all quadratic residues.  But the former then tells us $-1$ is not a quadratic residue.  Thus $\left(\frac{2}{p}\right) = 1$ and $\left(\frac{-1}{p}\right) = -1$, so by the two supplements to quadratic reciprocity we have $p \equiv 7~(\text{ mod }8)$.
\end{proof}

\begin{remark}
In the proof of Theorem \ref{thm:prime2dim} we applied the matroid basis exchange axiom only to a single term in each orbit (our particular basis $\{0,i\}$, and we only used our particular choice of $m$, in the notation of the proof).  These values were chosen carefully so that the two possible bases required by basis exchange are then in the \emph{same} orbit, which means we must have that particular orbit in our sum.  The advantage of doing this is that we don't need to break into cases, but the disadvantage is that in the end we are only able to guarantee we get all orbits this way when the numerical condition in Lemma \ref{lem:MO} holds.  We know that this numerical hypothesis is not actually necessary, as this hypothesis does not hold for $p=17$ and yet we checked using SAGE and found that the only two-dimensional representation for $\Z_{17}$ is the uniform matroid.  Thus, it is possible to weaken the hypothesis for this theorem but doing so would require applying basis exchange to other choices of $m$ and hence would require a complicated web of cases.  
\end{remark}

Based on the $p=17$ example alluded to above, it seems plausible to guess that Theorem \ref{thm:prime2dim} actually holds for all primes.  Since our $\Z_4$ and $\Z_6$ examples worked out in \S\ref{sec:ord4} and \S\ref{sec:ord6cyclic} suggest that non-uniform two-dimensional representations occur for non-prime order cyclic groups, we go further and venture the following:

\begin{conjecture}
There is only one two-dimensional subrepresentation of $\B[\Z_n]$ if and only if $n$ is prime.
\end{conjecture}

We do not know much about the higher-dimensional tropical representations of a cyclic group.  However, by focusing on \emph{realizable} representations we are able to make some progress.  First, note that the $n$ irreducible representations of the cyclic group $\Z_n$ are all given by powers of a primitive $n$-th root of unity, so that any $d$-dimensional subrepresentation of $\C[\Z_n]$ is given by a matrix of the form
\[\left(\begin{array}{ccccc}1 & \zeta^{i_1} & \zeta^{2i_1} & \cdots & \zeta^{(n-1)i_1} \\ \vdots & \vdots & \vdots & \ddots \\ 1 & \zeta^{i_d} & \zeta^{2i_d} & \cdots & \zeta^{(n-1)i_d}\end{array}\right)\]
for $I = \{i_1,\ldots,i_d\} \in \binom{\Z_n}{d}$.  The tropicalization of this representation is the matroid whose bases are the subsets $J\in\binom{\Z_n}{d}$ for which the minor in this matrix with columns $J$ is nonzero.  Since this does not depend on the order of the indices within $I$, we shall simply call this the \emph{$I$-th tropicalized representation} in $\B[\Z_n]$.

\begin{theorem}
For a cyclic group $\Z_p$ of prime order, the uniform subrepresentation in each dimension is realizable and these are the only realizable subrepresentations.
\end{theorem}

\begin{remark}
The second of these assertions can fail for cyclic groups of non-prime order; indeed, we show in \S\ref{sec:ord4} and \S\ref{sec:ord6cyclic} that there are non-uniform realizable representations of both $\Z_4$ and $\Z_6$ in dimension two.
\end{remark}

\begin{proof}
Consider the above matrix of primitive $n$-th roots of unity in the case $d=n$, so that we have a square matrix. Chebotarev's theorem on roots of unity is that if $n$ is prime, then all minors of this matrix are nonzero---so by selecting any $I$-rows we see that the $I$-th tropicalized representation is uniform.
\end{proof}

By slightly modifying the proof of \cite[Theorem 6]{Vandermonde} we are able to make some, though far less comprehensive, progress in the case of prime power order.  

\begin{theorem}
For $I,J\in \binom{\Z_{p^r}}{d}$, if the $J$-th Pl\"ucker coordinate of the $I$-th representation of $\Z_{p^d}$ is zero, then at least two elements of $I$ are equal mod $p$ and the same holds for $J$.
\end{theorem}

\begin{proof}
Following \cite{Vandermonde}, for $\vec{a} = (a_1,\ldots,a_d)$ a $d$-tuple of distinct non-negative integers and $\vec{x} = (x_1,\ldots,x_d)$ a $d$-tuple of indeterminates, let $V_{\vec{a}}(\vec{x})$ denote the determinant of the matrix with $(i,j)$-entry $x_i^{a_j}$.  Write $\vec{d} := (0,1,\ldots,d-1)$.  Then 
\[V_{\vec{d}}(\vec{x}) = \prod_{1 \le i < j \le d}(x_j-x_i)\]
is the usual Vandermonde determinant.  As in \cite{Vandermonde}, let us write \[P_{\vec{a}}(\vec{x}) := \frac{V_{\vec{a}}(\vec{x})}{V_{\vec{d}}(\vec{x})} \in \Z[x_1,\ldots,x_d].\]
Let us also write $x^{\vec{a}} := (x^{a_1},x^{a_2},\ldots,x^{a_n})$ where $x$ is a single indeterminate.  

Fix an order on the subsets $I,J\in\binom{\Z_{p^r}}{d}$.  Then, up to sign (which won't affect anything in this discussion), the $J$-th Pl\"ucker coordinate of the $I$-th representation of $\Z_{p^r}$ is $V_I(\zeta^J) = V_J(\zeta^I)$.  Since the powers $\zeta^J = (\zeta^{j_1},\ldots,\zeta^{j_d})$ are distinct, the Vandermonde determinant $V_{\vec{d}}(\zeta^J)$ is nonzero.  Thus, our hypothesis that the Pl\"ucker coordinate $V_I(\zeta^J) = V_{\vec{d}}(\zeta^J)P_I(\zeta^J)$ is zero implies $P_I(\zeta^J)=0$.  This implies that the minimal polynomial of $\zeta$, namely the cyclotomic polynomial $\Phi_{p^r}(x)$, divides the polynomial $P_I(x^J)$.  This next step is where we use that our cyclic group has prime power order: it is a standard and elementary fact in algebraic number theory (proven, for instance, using Mobius inversion) that $\Phi_{p^r}(1) = p$.  So $p$ divides $P_I(\vec{1})$.  It is proven in \cite{Vandermonde} that $P_I(\vec{1})V_{\vec{d}}(\vec{d}) = V_{\vec{d}}(I)$, so we see that $p$ divides $V_{\vec{d}}(I)$---and since this latter quantity is an ordinary Vandermonde determinant, we can conclude that at least two indices in $I$ must coincide mod $p$.  Reversing the role of $I$ and $J$ in this proof allows us to conclude the same about $J$.
\end{proof}


\subsection{Groups of small order}

Here we work out examples of the theory developed in this paper---specifically concerning subrepresentations of the regular representation---by going through some groups of small order.  First, recall that (1) the tropical linear space associated to the uniform matroid of any rank $\le |G|$ is a subrepresentation (Proposition \ref{prop:unif}); (2) for rank two and $G$ abelian, this subrepresentation is realizable if and only if $G$ is cyclic (Theorem \ref{thm:cyclic}); (3) the only dimension one and codimension one subrepresentations, for any finite $G$, are the tropical linear spaces associated to uniform matroids, both of which are always realizable (Proposition \ref{prop:dimone}).

\subsubsection{Order 4}\label{sec:ord4}

There are only two groups to consider, $\Z_4$ and $\Z_2\times \Z_2$, and there is only one interesting dimension to consider, namely dimension two.  Since $\Z_4$ is cyclic we already know that the tropical linear space associated to the uniform matroid $U_{2,\Z_4}$ is realizable, but are there any non-uniform subrepresentations?  Let us write the basis for $\B[\Z_4]$ as $e_0,e_1,e_2,e_3$ with the obvious relation to the group $\Z_4$.  A subrepresentation is a fixed point for the induced $\Z_4$-action on $\ext^2 \B[\Z_4]$ that also satisfies the tropical Pl\"ucker relation (Corollary \ref{cor:DressAct}).  A fixed point must be a sum of orbits, and there are only two orbits in $\ext^2 \B[\Z_4]$, namely \[e_{02}+e_{13}\text{ and }e_{01}+e_{12}+e_{23}+e_{03}.\]  The former of these is not a matroid, the latter is, and their sum corresponds to the uniform matroid.  Thus there is exactly one non-uniform subrepresentation; it is self-dual and spanned by $e_0+e_2$, $e_1+e_3$, and $e_i+e_j+e_k$ for all $\{i,j,k\}\in\binom{\Z_4}{3}$.  

This non-uniform subrepresentation is also realizable---for instance, it is realized by the direct sum of the trivial representation and the sign representation.  In fact, one can check that among the six subrepresentations in $\C[\Z_4]$ of dimension two (all such are direct sums of two of the four one-dimensional representations), two of them tropicalize to our non-uniform subrepresentation and four of them tropicalize to the uniform subrepresentation.  

\begin{remark}\label{rem:indecomp}
It is worth noting that this non-uniform two-dimensional subrepresentation is ``indecomposable'' in the sense that it is not the stable sum of two subrepresentations of lower dimension.  Indeed, stable sum corresponds to matroid union, and the only one-dimensional subrepresentation corresponds to $U_{1,\Z_4}$ and the matroid union of that with itself is $U_{2,\Z_4}$.  In fact, this shows that for any finite group $G$, any non-uniform two-dimensional subrepresentation of $\B[G]$ is indecomposable in this sense (higher dimension is more subtle/interesting, since for instance a four-dimensional representation might decompose as a stable sum of two indecomposable two-dimensional representations).  A word of caution, however: this notion of indecomposability is not intrinsic, it is a property of the embedding in the regular representation.
\end{remark}

Let us turn now to the Klein four group, and let us write the basis for $\B[\Z_2\times\Z_2]$ as $e_{(0,0)},e_{(1,0)},e_{(0,1)},e_{(1,1)}$.  Again we need only look in dimension two, and we do so by first writing down all the orbits in $\ext^2\B[\Z_2\times\Z_2]$ (here to avoid confusion we do not abbreviate the wedge notation): 
\begin{eqnarray*}
e_{(0,0)}\wedge e_{(1,0)} + e_{(0,1)}\wedge e_{(1,1)},\\
e_{(0,0)}\wedge e_{(0,1)} + e_{(1,0)}\wedge e_{(1,1)},\\
e_{(0,0)}\wedge e_{(1,1)} + e_{(1,0)}\wedge e_{(0,1)}.
\end{eqnarray*}
None of these orbits forms a matroid, but the sum of any two of them does---and of course the sum of all three forms the uniform matroid which we already know corresponds to a non-realizable subrepresentation since $\Z_2\times \Z_2$ is not cyclic---so there are exactly three non-uniform subrepresentations (all three of which are indecomposable as discussed in Remark \ref{rem:indecomp}).  One readily checks, for instance by explicitly writing down Pl\"ucker vectors, that all three of these are realizable.

\subsubsection{Order 5}

Here we only have the cyclic group $\Z_5$, and the dimensions to explore are two and three---but since the complement of a subrepresentation is a subrepresentation (Theorem \ref{thm:complement}), in order to classify all subrepresentations it suffices to consider dimension two.  By Theorem \ref{thm:prime2dim} the only two-dimensional representation is $U_{2,\Z_5}$, and since $\Z_5$ is cyclic we know by Theorem \ref{thm:cyclic} that this is realizable.

\subsubsection{Order 6 --- cyclic group}\label{sec:ord6cyclic}

For $\Z_6$, the orbits in $\ext^2\B[\Z_6]$ are
\begin{eqnarray*}
f_1 := e_{01} + e_{12} + e_{23} + e_{34} + e_{45} + e_{05},\\
f_2 := e_{02} + e_{13} + e_{24} + e_{35} + e_{04} + e_{15},\\
f_3 := e_{03} + e_{14} + e_{25}.
\end{eqnarray*}
None of these are matroids, but $f_1+f_2$ and $f_1+f_3$ are (as is the uniform matroid $f_1+f_2+f_3$, of course).   If $\zeta$ denotes a primitive sixth root of unity in $\C$, then any two-dimensional subrepresentation of $\C[\Z_6]$ is given by a matrix of the form
\[\left(\begin{array}{cccccc}1 & \zeta^j & \zeta^{2j} & \zeta^{3j} & \zeta^{4j} & \zeta^{5j} \\1 & \zeta^k & \zeta^{2k} & \zeta^{3k} & \zeta^{4k} & \zeta^{5k}\end{array}\right)\]
for $0 \le j \ne k \le 5$.  For $0 \le a \ne b \le 5$, the $(a,b)$-minor is then \[\zeta^{aj+bk} - \zeta^{ak+bj} = \zeta^{ak+bj}(\zeta^{(a-b)(j-k)} - 1).\]  Thus $\{a,b\}$ is a basis for the associated matroid/tropicalization if and only if
\[(a-b)(j-k) \not\equiv 0~(\text{mod }6).\]  If $k-j \equiv \pm 2$ then we get the matroid $f_1+f_2$, if $k-j\equiv 3$ then we get $f_1+f_3$, and in the remaining cases we get the uniform two-dimensional subrepresentation.

The $\Z_6$-orbits in $\ext^3\B[\Z_6]$ are
\begin{eqnarray*}
g_1 := e_{012} + e_{123} + e_{234} + e_{345} + e_{045} + e_{015},\\
g_2 := e_{013} + e_{124} + e_{235} + e_{034} + e_{145} + e_{025},\\
g_3 := e_{014} + e_{125} + e_{023} + e_{134} + e_{245} + e_{035},\\
g_4 := e_{024} + e_{135}.
\end{eqnarray*}
None of these are matroids, but we get the two matroids $g_1+g_4$ and $g_1+g_2+g_3$, in addition to the uniform matroid $U_{3,\Z_6}$ given by $\sum g_i$.  If interested, one could study realizability of these by looking at $3\times 3$ minors of sixth roots of unity; we shall skip this and instead turn to indecomposability in the sense of Remark \ref{rem:indecomp}.  The matroid union of $U_{1,\Z_6}$ with itself three times yields $U_{3,\Z_6}$, but now we must also consider the matroid union of $U_{1,\Z_6}$ with the rank two matroids $f_1+f_2$ and $f_1+f_3$ above (in general, matroid union with a uniform matroid is called matroid \emph{elongation} \cite[Proposition 5.3.1]{GG3}).  One easily checks that 
\[(f_1+f_2)\wedge \left(\sum e_i\right) = \sum g_i \text{ and }(f_1+f_3)\wedge \left(\sum e_i\right) = g_1+g_2+g_3\]
so we have exactly one indecomposable dimension three subrepresentation, namely $g_1+g_4$.

\subsubsection{Order 6 --- symmetric group}\label{sec:ord6symmetric}

We now turn to the other group of order six, the symmetric group on three letters.  We shall view this as the dihedral group 
\[D_3 = \{\rho,\sigma~|~\rho^3=\sigma^2=(\sigma\rho)^2=1\}\] and label the basis for $\B[D_3]$ as follows:
\[e_1,~e_{\sigma},~e_{\rho},~e_{\sigma\rho},~e_{\rho^2},~e_{\sigma\rho^2}.\]
The orbits in $\ext^2\B[D_3]$ are
\begin{eqnarray*}
f_1 := e_1\wedge e_\sigma + e_{\rho}\wedge e_{\sigma\rho^2} + e_{\rho^2}\wedge e_{\sigma\rho}\\
f_2 := e_1\wedge e_{\sigma\rho} + e_{\rho}\wedge e_{\sigma} + e_{\rho^2}\wedge e_{\sigma\rho^2}\\
f_3 := e_1\wedge e_{\sigma\rho^2} + e_{\rho}\wedge e_{\sigma\rho} + e_{\rho^2}\wedge e_{\sigma}\\
f_4 := e_1\wedge e_\rho + e_{\rho}\wedge e_{\rho^2} + e_{1}\wedge e_{\rho^2} + e_\sigma\wedge e_{\sigma\rho} + e_{\sigma\rho}\wedge e_{\sigma\rho^2} + e_{\sigma}\wedge e_{\sigma\rho^2}\\
\end{eqnarray*}
The linear combinations of these $f_i$ which form matroids are precisely the sums of any three of them, as well as the uniform matroid given by the sum of all four.

To determine realizability of these subrepresentations, we need to classify the two-dimensional subrepresentations of $\C[D_3]$.  Doing so up to isomorphism is a well-known staple of classical invariant theory, but we need to explicitly coordinatize everything.  The main subtlety is that in the isotypic module decomposition \[\C[D_3] \cong V_{\text{triv}}\oplus V_{\text{sign}}\oplus V_{\text{std}}^{\oplus 2}\] there are many different possible decompositions of the latter factor into two copies of the two-dimensional standard representation.  Since the one-dimensional representations are unique, let us start with those: the only reducible two-dimensional representation in $\C[D_3]$ is $V_{\text{triv}}\oplus V_{\text{sign}}$, which can be presented by the matrix
\[\left(\begin{array}{cccccc}1 & 1 & 1 & 1 & 1 & 1 \\1 & -1 & 1 & -1 & 1 & -1\end{array}\right)\]
and hence has associated tropicalization/matroid $f_1+f_2+f_3$.  

As for irreducible two-dimensional representations, they are all isomorphic to the standard representation and are all spanned by a pair of vectors $v,~\sigma\cdot v$ where $v$ is an eigenvector for $\rho$ with eigenvalue $\omega$ a primitive cube root of unity (it follows that $\sigma\cdot v$ is then an eigenvector for $\rho$ with eigenvalue $\omega^2$); see \cite[\S1.3]{Fulton-Harris}.  By parameterizing the eigenspace we see that any copy of the standard representation is given by
\[\left(\begin{array}{cccccc}z_1 & z_2 & \omega^2 z_1 & \omega z_2 & \omega z_1 & \omega^2 z_2 \\z_2 & z_1 & \omega z_2 & \omega^2 z_1 & \omega^2 z_2 & \omega z_1\end{array}\right)\]
for $z_1,z_2\in \C$ with $(z_1,z_2)\ne (0,0)$.  One can see directly from the Pl\"ucker vector of $2\times 2$ minors that the associated matroid is:
\begin{itemize}
\item $U_{2,D_3}$ if $z_1,z_2$ are sufficiently generic;
\item $f_2+f_3+f_4$ if $z_1^2 = z_2^2$;
\item $f_1+f_2+f_3$ if $z_1z_2=0$;
\item $f_1+f_2+f_4$ if $z_1^2 = \omega z_2^2$;
\item $f_1+f_3+f_4$ if $z_1^2 = \omega^2 z_2^2$.
\end{itemize}
In particular, notice that (1) the uniform two-dimensional subrepresentation is realizable, even though $D_3$ is not cyclic (contrasting the situation for abelian groups \ref{thm:cyclic})---in fact, \emph{every} two-dimensional subrepresentation is realizable for this group---and (2) the subrepresentation with tropical Pl\"ucker vector $f_1+f_2+f_3$ is realized by both a reducible subrepresentation, $V_{\text{triv}}\oplus V_{\text{sign}}$, \emph{and} by an irreducible one, namely a copy of $V_{\text{std}} \subseteq \C[D_3]$ for a sufficiently generic splitting of the isotypic module $V_{\text{std}}^{\oplus 2}$.  

We leave the study of dimension three subrepresentations of $\B[D_3]$ to the enthusiastic reader.

\subsection{Character tables}

We would like to define the character of a tropical subrepresentation $T \subseteq V$.  However, since tropical linear spaces generally are not free modules, we cannot simply restrict an automorphism of $V$ to $T$ and obtain a matrix from which to compute the trace.  Here we explore a potential remedy to this; for simplicity, we stick with the Boolean setting $V\cong \B^n$, but one could presumably extend this to arbitrary idempotent semifields if interested.  The key observation is that the matrix of a linear transformation of free modules records where a basis is sent, and the trace is independent of the choice of basis---and while a tropical linear space does not have a basis, it does admit a canonical generating set, namely the cocircuit vectors.

Fix a linear representation $G \rightarrow \GL(V)$ over $\B$ and a tropical subrepresentation $T \subseteq V$.  By Proposition \ref{prop:matroid}, each $g\in G$ yields an automorphism of the matroid $M(T)$ associated to $T$, so by Lemma \ref{lem:MatAut} we see that $G$ acts on the set of cocircuits of $M(T)$.

\begin{definition}
The \emph{character} $\chi_T : G \rightarrow \B$ of $T \subseteq V$ is the function sending $g\in G$ to the trace (sum of diagonal entries) of the permutation matrix coming from the action on the cocircuits of $M(T)$:
\[\chi_T(g) = \begin{cases} 
1 & \text{if $g$ fixes some cocircuit}\\
0 & \text{otherwise}. 
\end{cases}
\]
\end{definition}

Since for any group action, $g\in G$ fixes an element if and only if $hgh^{-1}$ fixes an element, we see immediately that these characters are class functions.  

It is not too clear what the analogue of a character table is, since we don't have a notion of irreducibility for tropical representations and we also don't just consider representations up to isomorphism.  We propose the following: the \emph{tropical character table} of a finite group $G$ is the table of character values whose columns are labelled by the conjugacy classes of $G$ and whose rows are labelled by the subrepresentations of the Boolean regular representation $\B[G]$.

Admittedly, we are not sure whether this is an appropriate and/or useful extension of character theory to the tropical setting.  We leave it as future work to uncover the significance of these characters and their tables, if there is such to be found; for now we simply compute some examples.  In what follows, we use the notation from the preceding classification of subrepresentations for groups of small order.

\[
\left[\begin{array}{c|cccc}\Z_4 & 0 & 1 & 2 & 3 \\ \hline 
U_{1,\Z_4} & 1 & 1 & 1 & 1 \\
e_{01}+e_{12}+e_{23}+e_{03} & 1 & 0 & 1 & 0 \\
U_{2,\Z_4} & 1 & 0 & 0 & 0 \\
U_{3,\Z_4} & 1 & 0 & 1 & 0
\end{array}\right]
\]

\vspace{0.1in}

\[
\left[\begin{array}{c|cccc}\Z_2\times\Z_2 & (0,0) & (1,0) & (0,1) & (1,1) \\ \hline 
U_{1,\Z_2\times\Z_2} & 1 & 1 & 1 & 1 \\
e_{(0,0)}\wedge e_{(1,0)} + e_{(0,1)}\wedge e_{(1,1)} + e_{(0,0)}\wedge e_{(0,1)} + e_{(1,0)}\wedge e_{(1,1)} & 1 & 0 & 0 & 1 \\
e_{(0,0)}\wedge e_{(1,0)} + e_{(0,1)}\wedge e_{(1,1)} + e_{(0,0)}\wedge e_{(1,1)} + e_{(1,0)}\wedge e_{(0,1)} & 1 & 0 & 1 & 0 \\
e_{(0,0)}\wedge e_{(0,1)} + e_{(1,0)}\wedge e_{(1,1)} + e_{(0,0)}\wedge e_{(1,1)} + e_{(1,0)}\wedge e_{(0,1)} & 1 & 1 & 0 & 0 \\
U_{2,\Z_2\times\Z_2} & 1 & 0 & 0 & 0 \\
U_{3,\Z_2\times\Z_2} & 1 & 1 & 1 & 1 \\
\end{array}\right]
\]

\vspace{0.1in}

\[
\left[\begin{array}{c|ccccc}\Z_5 & 0 & 1 & 2 & 3 & 4\\ \hline 
U_{1,\Z_5} & 1 & 1 & 1 & 1 & 1\\
U_{2,\Z_5} & 1 & 0 & 0 & 0 & 0\\
U_{3,\Z_5} & 1 & 0 & 0 & 0 & 0\\
U_{4,\Z_5} & 1 & 0 & 0 & 0 & 0\\
\end{array}\right]
\]

\vspace{0.1in}

Let us work through the example of $\Z_2\times \Z_2$ to show how these computations are done.  First, recall that the cocircuits of a matroid are the minimal dependent sets of the dual matroid (and the bases of the dual matroid are the complements of the bases of the original matroid).  For any group $G$ the only dimension one subrepresentation is $U_{1,G}$, which has dual matroid $U_{|G|-1,G}$ and hence a unique cocircuit $\sum_{g\in G}e_g$; since this cocircuit is fixed by all $g\in G$, the first row of any tropical character table is all 1s.  (The first column is trivially all 1s since the identity element clearly fixes a cocircuit.)  The matroid $U_{2,\Z_2\times\Z_2}$ is self-dual so its cocircuits are the sums of triples of $e_g$; one then checks that no triple is fixed by any non-identity element.  On the other hand, the dual of $U_{3,\Z_2\times\Z_2}$ is $U_{1,\Z_2\times\Z_2}$ so the cocircuits are the sums of pairs, and we have that
\begin{eqnarray*}
e_{(1,0)}\text { fixes }e_{(0,1)} + e_{(1,1)},\\
e_{(0,1)}\text{ fixes }e_{(1,0)} + e_{(1,1)},\\
e_{(1,1)}\text{ fixes }e_{(0,0)} + e_{(1,1)},
\end{eqnarray*}
so this row of the table is all 1s.  Next, consider the first of the non-uniform matroid rows listed above in the table for $\Z_2\times \Z_2$.  This matroid is self-dual, and the cocircuits are $e_{(0,0)} + e_{(1,1)}$ and $e_{(0,1)} + e_{(1,0)}$, as well as all triples that do not contain either of these pairs.  We already noted that none of the triples are fixed by a non-identity element, so we just need to check these two size two cocircuits: $e_{(0,1)}$ does not fix either of them, nor does $e_{(1,0)}$, but $e_{(1,1)}$ does fix $e_{(0,0)} + e_{(1,1)}$.  The next non-uniform matroid row in the table is also self-dual, with cocircuits $e_{(0,0)} + e_{(0,1)}$ and $e_{(1,0)} + e_{(1,1)}$ in addition to the triples not containing these; the first of these is fixed by $e_{(0,1)}$.  Finally, the third non-uniform matroid row is obtained from the second by applying the involution that swaps the first and second factors of $\Z_2\times \Z_2$.

\bibliographystyle{amsalpha}
\bibliography{bib}
\end{document}